\pdfoutput=1
\documentclass[a4paper, reqno]{amsart}
\usepackage[utf8]{inputenc}
\usepackage{graphicx}

\usepackage[UKenglish]{babel}

\DeclareSymbolFont{symbolsC}{U}{pxsyc}{m}{n}
\DeclareMathSymbol{\coloneqq}{\mathrel}{symbolsC}{"42}

\usepackage{comment}

\usepackage{pythonhighlight}

\usepackage{csquotes} 
 
\usepackage{float}

\usepackage[
backend=bibtex,
style=alphabetic,
sorting=nyt
]{biblatex}
\title{On the number of Enriques quotients for supersingular K3 surfaces}
\author{Kai Behrens}

\usepackage{amsthm,amsfonts,amssymb,amsmath}
\usepackage[all]{xy}
\usepackage{xr-hyper}
\usepackage{verbatim}
\usepackage{pdfsync}
\usepackage{pigpen}
\usepackage{stmaryrd}
\usepackage{mathtools}

\newcommand{\CA}{{\mathcal {A}}}

\newcommand{\CC}{{\mathcal {C}}}

\newcommand{\CE}{{\mathcal {E}}}

\newcommand{\CM}{{\mathcal {M}}}

\newcommand{\CO}{{\mathcal {O}}}

\newcommand{\CX}{{\mathcal {X}}}
\newcommand{\CY}{{\mathcal {Y}}}

\newcommand{\Aut}{{\mathrm{Aut}}}

\newcommand{\disc}{{\mathrm{disc}}}

\newcommand{\lra}{\longrightarrow}

\newcommand{\po}{\ar@{}[dr]|{\text{\pigpenfont R}}}
\newcommand{\pb}{\ar@{}[dr]|{\text{\pigpenfont J}}}

\newtheorem{theorem}{Theorem}[section]
\newtheorem{proposition}[theorem]{Proposition}
\newtheorem{lemma}[theorem]{Lemma}

\newtheorem{corollary}[theorem]{Corollary}

\newtheorem*{theorem*}{Theorem}
\newtheorem*{proposition*}{Proposition}

\theoremstyle{definition}
\newtheorem{definition}[theorem]{Definition}

\newtheorem{remark}[theorem]{Remark}

\numberwithin{equation}{section}

\usepackage{tikz}
\usepackage{verbatim}
\usetikzlibrary{arrows,chains,matrix,positioning,scopes}
\makeatletter
\tikzset{join/.code=\tikzset{after node path={%
\ifx\tikzchainprevious\pgfutil@empty\else(\tikzchainprevious)%
edge[every join]#1(\tikzchaincurrent)\fi}}}
\makeatother
\tikzset{>=stealth',every on chain/.append style={join},
         every join/.style={->}}
\tikzstyle{labeled}=[execute at begin node=$\scriptstyle,
   execute at end node=$]

\makeatletter
\let\save@mathaccent\mathaccent
\newcommand*\if@single[3]{%
  \setbox0\hbox{${\mathaccent"0362{#1}}^H$}%
  \setbox2\hbox{${\mathaccent"0362{\kern0pt#1}}^H$}%
  \ifdim\ht0=\ht2 #3\else #2\fi
  }
\newcommand*\rel@kern[1]{\kern#1\dimexpr\macc@kerna}
\newcommand*\widebar[1]{\@ifnextchar^{{\wide@bar{#1}{0}}}{\wide@bar{#1}{1}}}
\newcommand*\wide@bar[2]{\if@single{#1}{\wide@bar@{#1}{#2}{1}}{\wide@bar@{#1}{#2}{2}}}
\newcommand*\wide@bar@[3]{%
  \begingroup
  \def\mathaccent##1##2{%
    \let\mathaccent\save@mathaccent
    \if#32 \let\macc@nucleus\first@char \fi
    \setbox\z@\hbox{$\macc@style{\macc@nucleus}_{}$}%
    \setbox\tw@\hbox{$\macc@style{\macc@nucleus}{}_{}$}%
    \dimen@\wd\tw@
    \advance\dimen@-\wd\z@
    \divide\dimen@ 3
    \@tempdima\wd\tw@
    \advance\@tempdima-\scriptspace
    \divide\@tempdima 10
    \advance\dimen@-\@tempdima
    \ifdim\dimen@>\z@ \dimen@0pt\fi
    \rel@kern{0.6}\kern-\dimen@
    \if#31
      \overline{\rel@kern{-0.6}\kern\dimen@\macc@nucleus\rel@kern{0.4}\kern\dimen@}%
      \advance\dimen@0.4\dimexpr\macc@kerna
      \let\final@kern#2%
      \ifdim\dimen@<\z@ \let\final@kern1\fi
      \if\final@kern1 \kern-\dimen@\fi
    \else
      \overline{\rel@kern{-0.6}\kern\dimen@#1}%
    \fi
  }%
  \macc@depth\@ne
  \let\math@bgroup\@empty \let\math@egroup\macc@set@skewchar
  \mathsurround\z@ \frozen@everymath{\mathgroup\macc@group\relax}%
  \macc@set@skewchar\relax
  \let\mathaccentV\macc@nested@a
  \if#31
    \macc@nested@a\relax111{#1}%
  \else
    \def\gobble@till@marker##1\endmarker{}%
    \futurelet\first@char\gobble@till@marker#1\endmarker
    \ifcat\noexpand\first@char A\else
      \def\first@char{}%
    \fi
    \macc@nested@a\relax111{\first@char}%
  \fi
  \endgroup
}
\makeatother

\begin{document}
\begin{abstract} We show that most classes of K3 surfaces have only finitely many Enriques quotients up to isomorphism. For supersingular K3 surfaces over fields of characteristic $p \geq 3$, we give a formula which generically yields the number of their Enriques quotients. We reprove via a lattice theoretic argument that supersingular K3 surfaces always have an Enriques quotient over fields of small characteristic. For some small characteristics and some Artin invariants, we explicitly compute lower bounds for the number of Enriques quotients of a supersingular K3 surface. We show that the supersingular K3 surface of Artin invariant $1$ over an algebraically closed field of characteristic $3$ has exactly two Enriques quotients.
\end{abstract}
\maketitle
\section*{Introduction}
If $X$ is a K3 surface over an arbitrary field $k$ and $\iota \colon X \lra X$ is an involution without fixed points, then the quotient variety $X/ \langle \iota \rangle$ is an Enriques surface. For any Enriques surface $Y$ over a field of characteristic $p \neq 2$ there exists (up to isomorphism) a unique K3 surface $X$ such that $Y$ is isomorphic to such a quotient $X / \langle \iota \rangle$. In other words, any Enriques surface has a unique K3 cover. We may now ask, given a K3 surface $X$, how many isomorphism classes of Enriques surfaces $Y$ there are, such that there exists a fixed point free involution $\iota \colon X \rightarrow X$ and an isomorphism $Y \cong X/ \langle \iota \rangle$.

For complex K3 surfaces, there exists a Torelli theorem in terms of Hodge cohomology \cite{MR0284440}, \cite{MR0447635}. If $Y$ is an Enriques surface, then its Neron-Severi group $\mathrm{NS}(Y)$ is isomorphic to the quadratic form $\Gamma'= \Gamma \oplus \mathbb{Z}/2\mathbb{Z}$ with $\Gamma = U_2 \oplus E_8(-1)$.  By the Torelli theorem for complex K3 surfaces, fixed point free involutions of a K3 surface $X$ can then be characterized in terms of primitive embeddings $\Gamma(2) \hookrightarrow \mathrm{NS}(X)$ without vectors of self-intersection $-2$ in the complement of $\Gamma(2)$. Denoting the set of all such embeddings by $\mathfrak{M}$, Ohashi \cite{MR2319542} used this connection to prove the following formula, which yields an upper bound for the number of isomorphism classes of Enriques quotients of any complex K3 surface and is an equality for generic K3 surfaces. 
\begin{theorem*}\cite[Theorem 2.3]{MR2319542} Let $X$ be a complex K3 surface. By $q_{\mathrm{NS}(X)}$ we denote the discriminant form of the Neron-Severi group of $X$. Let $M_1, \ldots, M_k \in \mathfrak{M}$ be a complete set of representatives for the action of $O(\mathrm{NS}(X))$ on $\mathfrak{M}$. For each $j \in \{1, \ldots, k\}$, we let
\begin{align*} K^{(j)}=\{\psi \in O(\mathrm{NS}(X)) \mid \psi(M_j)=M_j \}\end{align*}
be the stabilizer of $M_j$ and $\mathrm{pr}(K^{(j)})$ be its canonical image in $O(q_{\mathrm{NS}(X)})$. Then we have an inequality
\begin{align*} \#\left\{\text{Enriques quotients of } X \right\} \leq \sum_{j=1}^k \#\left(O\left(q_{\mathrm{NS}(X)}\right)/\mathrm{pr}\left(K^{(j)}\right)\right). \end{align*}
If $X$ is such that the canonical morphism $\psi \colon O(\mathrm{NS}(X)) \rightarrow O\left(q_{\mathrm{NS}(X)}\right)$ is surjective and for each automorphism $\theta \in \mathrm{Aut}(X)$ the induced automorphism on the quotient $\mathrm{NS}(X)^{\vee}/\mathrm{NS}(X)$ is either the identity or multiplication by $-1$, then the inequality above becomes an equality.\end{theorem*}

In particular, it follows from the theorem above that the number of Enriques quotients of a complex K3 surface is finite.

We now want to understand the situation for K3 surfaces over fields of positive characteristic. Some of our results might already be known to the experts, but we could not find them in the literature. 

In Section \ref{lift} of this article we observe that the following statement follows directly from results of Lieblich and Maulik \cite{2011arXiv1102.3377L}.

\begin{theorem*}[see Theorem \ref{finht}] Let $X$ be a K3 surface over an algebraically closed field $k$. If $X$ is of finite height, then the number of isomorphism classes of Enriques quotients of $X$ is finite. \end{theorem*}

For many K3 surfaces of finite height, there exist special lifts to characteristic zero, which allow to compare their Enriques involutions. In particular, the situation for K3 surfaces of finite height should be very similar to the situation in characteristic zero and we refer to Remark \ref{near0} for details.

In view of these results, we then turn our focus towards Enriques quotients of (Shioda-) supersingular K3 surfaces over fields of characteristic $p \geq 3$. Ogus proved a Torelli-type theorem for supersingular K3 surfaces in terms of Crystalline cohomology \cite{MR717616} over fields of characteristic $p \geq 5$ and in light of recent results by Bragg and Lieblich \cite[Section 5.1]{2018arXiv180407282B} his proof also works over characteristic $p=3$, and we can therefore prove a formula for an upper bound of Enriques quotients of a supersingular K3 surface analogously to the results by Ohashi in the complex case. 
\begin{theorem*}[see Theorem \ref{form}]  Let $k$ be an algebraically closed field of characteristic $p \geq 3$ and let $X$ be a supersingular K3 surface over $k$. By $q_{\mathrm{NS}(X)}$ we denote the discriminant form of the Neron-Severi group of $X$. Let $M_1, \ldots, M_k \in \mathfrak{M}$ be a complete set of representatives for the action of $O(\mathrm{NS}(X))$ on $\mathfrak{M}$. For each $j \in \{1, \ldots, k\}$, we let
\begin{align*} K^{(j)}=\{\psi \in O(\mathrm{NS}(X)) \mid \psi(M_j)=M_j \}\end{align*}
be the stabilizer of $M_j$ and $\mathrm{pr}(K^{(j)})$ be its canonical image in $O(q_{\mathrm{NS}(X)})$. Then we have inequalities
\begin{align*} k \leq \#\left\{\text{Enriques quotients of } X \right\} \leq \sum_{j=1}^k \#\left(O\left(q_{\mathrm{NS}(X)}\right)/\mathrm{pr}\left(K^{(j)}\right)\right). \end{align*}
If $X$ is such that for each automorphism $\theta \in \mathrm{Aut}(X)$ the induced automorphism on the quotient $\mathrm{NS}(X)^{\vee}/\mathrm{NS}(X)$ is either the identity or multiplication by $-1$, then the inequality above becomes an equality on the right side. \end{theorem*}
It essentially follows from results of Nygaard \cite{Nygaard1980} that the formula yields an equality on the right side in the generic case.

We then turn towards applications. The following result is due to Jang \cite{MR3350105}.
\begin{theorem*}\cite[Corollary 2.4]{MR3350105} Let $X$ be a supersingular K3 surface over an algebraically closed field $k$ of characteristic $p \geq 3$. Then $X$ has an Enriques quotient if and only if the Artin invariant $\sigma$ of $X$ is at most $5$. \end{theorem*}
The proof of the above proposition uses lifting to characteristic zero. In an earlier article \cite{2013arXiv1301.1118J} Jang proved the following weaker version of the proposition via a lattice theoretic argument.
\begin{proposition*}\cite[Theorem 4.5, Proposition 3.5]{2013arXiv1301.1118J} Let $k$ be an algebraically closed field of characteristic $p$ and let $X$ be a supersingular K3 surface of Artin invariant $\sigma$. If $\sigma=1$, then $X$ has an Enriques involution. If $\sigma \in \{3,5\}$, and $p=11$ or $p \geq 19$, then $X$ has an Enriques involution. If $\sigma \in \{2,4\}$, and $p = 19$ or $p \geq 29$, then $X$ has an Enriques involution. If $\sigma \geq 6$, then $X$ has no Enriques involution. \end{proposition*}
The proof boils down to the following: if $X$ is a supersingular K3 surface of Artin invariant $\sigma \leq 5$, we need to show that there exists a primitive embedding of lattices $\Gamma(2) \hookrightarrow \mathrm{NS}(X)$ without any vector of self-intersection $-2$ in the complement of $\Gamma(2)$. Jang proved that such embeddings exist when the characteristic of the base field is large enough, but the same argument does not work over fields of small characteristic. With the help of the algebra software \textsc{magma} we explicitly show that such embeddings exist in the remaining cases. Hence, our results combined with Jang's yield a new proof for \cite[Corollary 2.4]{MR3350105} which does not rely on previous results over fields of characteristic zero.

Having established that the set of isomorphism classes of Enriques quotients of a supersingular K3 surface $X$ of Artin invariant $\sigma \leq 5$ is always nonempty, we are now interested in calculating some explicit numbers. In practice it turns out that this is a hard problem, however when the characteristic $p$ of the ground field is small, we found the following lower bounds with the help of \textsc{magma}.
\begin{proposition*}[see Proposition \ref{lowb}] For the number of isomorphism classes of Enriques quotients of a supersingular K3 surface of Artin invariant $\sigma$ over an algebraically closed ground field $k$ of characteristic $p$ we found the following numbers of equivalence classes under the action of $O(\mathrm{NS}(X))$ on $\mathfrak{M}$ denoted by $\mathrm{Rep}(p, \sigma)$: 
\begin{table}[H]
\caption{Some results for the lower bounds $\mathrm{Rep}(p, \sigma)$}
\centering 
\begin{tabular}{c c c c c c}
\hline
\hline                        
$p$ & $\sigma=1$ & $\sigma=2$ & $\sigma=3$ & $\sigma=4$ & $\sigma=5$ \\
 [0.5ex]
\hline
$3$ & $2$ & $12$ & $30$ & $20$ & $7$ \\
$5$ & $10$ & $222$ & $875$ & $302$ & $24$  \\
$7$ & $42$ & $3565$  & $?$ & $4313$ & $81$ \\
$11$ & $256$ & $?$ & $?$ & $?$ & $438$ \\
$13$ & $537$ & $?$ & $?$ & $?$ & $866$\\
$17$ & $2298$ & $?$ & $?$ & $?$ & $2974$ \\
 [1ex]
\hline
\end{tabular}
\end{table} \end{proposition*}
Using the computer algebra program \textsc{sage} we then computed the cardinalities of the group quotients $O\left(q_{\mathrm{NS}(X)}\right)/\mathrm{pr}\left(K^{(j)}\right)$ in these cases and found the following results for the upper bounds in Theorem \ref{form}. For $\sigma > 1$ these yield the number of isomorphism classes of Enriques quotients for a generic supersingular K3 surface of Artin invariant $\sigma$.
\begin{proposition*}
For the number of isomorphism classes of Enriques quotients of a supersingular K3 surface of Artin invariant $\sigma$ over an algebraically closed ground field $k$ of characteristic $p$ we found the following upper bounds. When $\sigma > 2$, then these are the numbers of isomorphism classes of Enriques quotients of a general supersingular K3 surface of Artin invariant $\sigma$.
\begin{table}[H]
\caption{Some results for the upper bounds}
\centering 
\begin{tabular}{c c c c c }
\hline
\hline                        
$p$ & $\sigma=1$ & $\sigma=2$ & $\sigma=3$ & $\sigma=4$ \\
 [0.5ex]
\hline
$3$ & $2$ & $490$ & $1278585$ & $24325222428$ \\
$5$ & $33$ & $635765$ & $1614527971875$ & $37184780652626927616$  \\
$7$ & $175$ & $191470125$  & $?$ & $88339146755283817573908480$  \\
$11$ & $2130$ & $?$ & $?$ & $?$ \\
$13$ & $5985$ & $?$ & $?$ & $?$ \\
$17$ & $36000$ & $?$ & $?$ & $?$  \\
 [1ex]
\hline
\end{tabular}
\end{table}

\begin{table}[H]
\caption{Some results for the upper bounds}
\centering 
\begin{tabular}{c c}
\hline
\hline                        
$p$ & $\sigma=5$ \\
 [0.5ex]
\hline
$3$  & $1286212218643287$ \\
$5$  & $1300418157436546004702724096
$  \\
$7$ &  $146385612443146033546182607153135616$ \\
$11$ & $9360899237983480445308665427637667976947171328$ \\
$13$ & $86881471802459725997082069598436845809167673327616$\\
$17$ & $117559509833496435964143968964217511931559374134458712064$ \\
 [1ex]
\hline
\end{tabular}
\end{table}

\end{proposition*}
The situation in the case where $p=3$ and $\sigma=1$ is particularly easy and we observe the following result.
\begin{theorem*}[see Theorem \ref{p3s1}] There are exactly two isomorphism classes of Enriques quotients of the supersingular K3 surface $X$ of Artin invariant $1$ over an algebraically closed field $k$ of \mbox{characteristic $3$}.
\end{theorem*}
In particular, we can be explicit about these two Enriques quotients: they are the two Enriques surfaces with finite automorphism group of type III and IV, see Corollary \ref{types}.

In the case of singular complex K3 surfaces and their Enriques quotients similar computations have recently been done by Shimada and Veniani \cite{shimada2019enriques}.
\section*{Acknowledgements}
I thank my doctoral advisor Christian Liedtke for his extensive support of my work. I would also like to thank Markus Kirschmer for helping me with using the computer algebra program \textsc{Magma}. Further thanks go to Gebhard Martin for useful comments and remarks. I also thank Paul Hamacher for many helpful discussions. The author is supported by the ERC Consolidator Grant 681838 "K3CRYSTAL".
\section{Prerequisites and notation}\label{overview}
In this section we fix some notation and recall known results.

Let $k$ be a perfect field of characteristic $p \geq 3$. A K3 surface $X$ over $k$ is called \emph{(Shioda-) supersingular} if and only if $\mathrm{rk}(\mathrm{NS}(X)) = 22$. This definition of supersingularity is due to Shioda. There is a second definition for supersingularity due to Artin. Namely, a K3 surface $X$ over $k$ is called \emph{Artin supersingular} if and only if its formal Brauer group $\Phi^2_X$ is of infinite height. It follows from the Tate conjecture, that over any perfect field $k$ a K3 surface is Artin supersingular if and only if it is Shioda supersingular \cite{MR3265555}. Charles first proved the Tate conjecture over fields of characteristic at least $5$ \cite{MR3103257}. Using the Kuga-Satake construction, Madapusi Pera gave a proof of the Tate conjecture over fields of characteristic at least $3$ \cite{MR3370622}. Over fields of characteristic $p=2$, the Tate conjecture was proved by Kim and Madapusi Pera \cite{MR3569319}.
\subsection{Lattices}
We fix some notation and recall basic definitions and results on lattices from \cite{1980IzMat..14..103N}. 

In the following, by a \emph{lattice} $(L, \langle\cdot,\cdot\rangle)$ we mean a free $\mathbb{Z}$-module $L$ of finite rank together with a nondegenerate symmetric bilinear form $\langle \cdot, \cdot \rangle \colon L \times L \rightarrow \mathbb{Z}$. A morphism of lattices is a morphism of the underlying $\mathbb{Z}$-modules that is compatible with intersection forms. To simplify notation, we will often talk about the lattice $L$, omitting the bilinear form. The lattice $L$ is called \emph{even} if $\langle x, x \rangle \in \mathbb{Z}$ is even for each $x \in L$. A lattice is \emph{odd} if it is not even. For $a \in \mathbb{Q}$ and if $a \langle L, L \rangle \subset \mathbb{Z}$, we denote by $L(a)$ the twisted lattice with underlying $\mathbb{Z}$-module $L$ and bilinear form $a \langle \cdot, \cdot \rangle$.

After choosing a basis $\{e_1,\ldots, e_n\}$ of $L$, the \emph{discriminant of $L$} is defined to be $\disc L = \det \left(\left( e_i \cdot e_j \right)_{ij}\right) \in \mathbb{Z}$. This definition does not depend on the chosen basis. The lattice $L$ is called \emph{unimodular} if $\left| \disc L \right|= 1$. The \emph{dual lattice of $L$} is the free $\mathbb{Z}$-module $L^{\vee}= \mathrm{Hom}(L, \mathbb{Z}) \subseteq L \otimes \mathbb{Q}$ together with the bilinear form $\langle \cdot, \cdot \rangle_{L^{\vee}} \colon L^{\vee} \times L^{\vee} \rightarrow \mathbb{Q}$ induced from $\langle \cdot, \cdot \rangle_{\mathbb{Q}} \colon L_{\mathbb{Q}} \times L_{\mathbb{Q}} \rightarrow \mathbb{Q}$. The \emph{discriminant group $A_L = L^{\vee}/L$} is a finite abelian group and is equipped with a canonical finite quadratic form $q_L \colon A_L \rightarrow \mathbb{Q}/2\mathbb{Z}$ induced from $\langle \cdot, \cdot \rangle$. One can show that $\# A_L  = \disc L$. Let $p$ be a prime number. If $p \cdot A_L = 0$ we say that the lattice $L$ is \emph{$p$-elementary}. 

We define the \emph{signature of $L$} to be the signature $(l_+, l_-)$ of the quadratic space $L \otimes \mathbb{Q}$. Likewise, the lattice $L$ is called \emph{positive definite} (respectively \emph{negative definite}) if and only if the quadratic space $L \otimes \mathbb{Q}$ is \emph{positive definite} (respectively \emph{negative definite}). There are two lattices which we will frequently use within this work. Namely, we will write $U$ for the even unimodular lattice of signature $(1,1)$ and $E_8$ for the even unimodular lattice of signature $(0,8)$. It is well-known that by prescribing these invariants the lattices $U$ and $E_8$ are well-defined up to isomorphism.

Let us now turn to morphisms of lattices. It is easy to see that any morphism $\psi \colon L_1 \rightarrow L_2$ of lattices is automatically injective. We will therefore also use the term \emph{embedding of lattices} when talking about morphisms. A given embedding of lattices $\psi \colon L_1 \hookrightarrow L_2$ is called \emph{primitive} if the quotient $L_2/\psi(L_1)$ is a free $\mathbb{Z}$-module. On the other hand, if the quotient $L_2/\psi(L_1)$ is finite, then we call $L_2$ an \emph{overlattice of $L_1$}. 

To a lattice $L$ we associate its \emph{genus $[L]$}, which is the class consisting of all lattices $L'$ such that $L \otimes \mathbb{Z}_p \cong L' \otimes \mathbb{Z}_p$ for all primes $p$ and $L \otimes \mathbb{R} \cong L' \otimes \mathbb{R}$. We will use the following characterization of the genus of a lattice which is due to \cite[Corollary 1.9.4]{1980IzMat..14..103N}.
\begin{proposition}\label{genus} Let $L$ be an even lattice. Then the genus $[L]$ is uniquely determined by the signature of $L$ and the discriminant form $q_L$. \end{proposition}
There is also a version of Proposition \ref{genus} for the odd case. We will only need the even case in this work though and therefore omit the odd version. 
\subsection{K3 crystals}
Most of the following content is due to Ogus \cite{MR563467}\cite{MR717616}. A strong inspiration for our treatment in this section and a good source for the interested reader is \cite{MR3524169}.

For the definition of $F$-crystals and their slopes we refer to \cite[Chapter I.1]{MR563463}. Given a supersingular K3 surface $X$, it turns out that a lot of information is encoded in its second crystalline cohomology. We say that $H^2_{\mathrm{crys}}(X/W)$ is a \emph{supersingular K3 crystal} of rank $22$ in the sense of the following definition, due to Ogus \cite{MR563467}.
\begin{definition} Let $k$ be a perfect field of positive characteristic $p$ and let $W=W(k)$ be its Witt ring with lift of Frobenius $\sigma \colon W \rightarrow W$. A \emph{supersingular K3 crystal of rank $n$} over $k$ is a free $W$-module $H$ of rank $n$ together with an injective  $\sigma$-linear map 
\begin{align*} \varphi \colon H \rightarrow H, \end{align*}
 i.e.\ $\varphi$ is a morphism of abelian groups and $\varphi(a \cdot m)= \sigma(a) \cdot \varphi(m)$ for all $a \in W$ and $m \in H$, and a symmetric bilinear form
 \begin{align*}
 \langle -,- \rangle \colon H \times H \rightarrow W, 
 \end{align*}
such that
\begin{enumerate}
\item $p^2H \subseteq \mathrm{im}(\varphi)$,
\item the map $\varphi \otimes_W k$ is of rank $1$,
\item $\langle -,- \rangle$ is a perfect pairing,
\item $\langle \varphi(x), \varphi(y) \rangle= p^2\sigma \left(\langle x,y \rangle \right)$, and
\item the $F$-crystal $(H,\varphi)$ is purely of slope $1$.
\end{enumerate}
\end{definition}
The \emph{Tate module} $T_H$ of a K3 crystal $H$ is the $\mathbb{Z}_p$-module 
$$ T_H \coloneqq \{x \in H \mid \varphi(x)= px\}.$$
One can show that if $H=H^2_{\mathrm{crys}}(X/W)$ is the second crystalline cohomology of a supersingular K3 surface $X$ and $c_1 \colon \mathrm{Pic}(X) \rightarrow H^2_{\mathrm{crys}}(X/W)$ is the first crystalline Chern class map, we have $c_1(\mathrm{Pic}(X)) \subseteq T_H$. If $X$ is defined over a perfect field, the Tate conjecture is known, see \cite{MR3103257} \cite{MR3370622}, and it follows that we even have the equality $c_1(\mathrm{NS}(X)) \otimes \mathbb{Z}_p= T_H$. The following proposition on the structure of the Tate module of a supersingular K3 crystal is due to Ogus \cite{MR563467}.
\begin{proposition}
Let $(H,\varphi, \langle -,- \rangle)$ be a supersingular K3 crystal over a field $k$ of characteristic $p > 2$ and let $T_H$ be its Tate module. Then $\mathrm{rk}_W H= \mathrm{rk}_{\mathbb{Z}_p} T_H$ and the bilinear form $(H, \langle -,- \rangle)$ induces a non-degenerate form $T_H \times T_H \rightarrow \mathbb{Z}_p$ via restriction to $T_H$ which is not perfect. More precisely, we find
\begin{enumerate}
\item $\mathrm{ord}_p(A_{T_H})=2 \sigma$ for some positive integer $\sigma$,
\item $(T_H, \langle -,-\rangle)$ is determined up to isometry by $\sigma$,
\item $\mathrm{rk}_W H \geq 2 \sigma$ and
\item there exists an orthogonal decomposition
\begin{align*}
(T_H, \langle -,- \rangle) \cong (T_0,p \langle -,- \rangle) \perp (T_1, \langle -,- \rangle),
\end{align*}
where $T_0$ and $T_1$ are $\mathbb{Z}_p$-lattices with perfect bilinear forms and of ranks $\mathrm{rk}T_0= 2 \sigma$ and $\mathrm{rk}T_1= \mathrm{rk}_WH-2 \sigma$.
\end{enumerate}
\end{proposition}
The positive integer $\sigma$ is called the \emph{Artin invariant} of the K3 crystal $H$ \cite{MR563467}. When $H$ is the second crystalline cohomology of a supersingular K3 surface $X$, we have $1 \leq \sigma(H) \leq 10$.
\subsection{K3 lattices}
The previous subsection indicates that the Néron-Severi lattice $\mathrm{NS}(X)$ of a supersingular K3 surface $X$ plays an important role in the study of supersingular K3 surfaces via the first Chern class map. We say that $\mathrm{NS}(X)$ is a \emph{supersingular K3 lattice} in the sense of the following definition due to Ogus \cite{MR563467}.
\begin{definition} A \emph{supersingular K3 lattice} is an even lattice $(N,\langle -,-\rangle)$ of rank $22$ such that
\begin{enumerate}
\item the discriminant $d(N \otimes_{\mathbb{Z}} \mathbb{Q})$ is $-1$ in $\mathbb{Q}^{\ast}/{\mathbb{Q}^{\ast}}^2$,
\item the signature of $N$ is $(1,21)$, and
\item the lattice $N$ is $p$-elementary for some prime number $p$.
\end{enumerate}
\end{definition} 
When $N$ is the Néron-Severi lattice of a supersingular K3 surface $X$, then the prime number $p$ in the previous definition turns out to be the characteristic of the base field. One can show that if $N$ is a supersingular K3 lattice, then its discriminant is of the form $d(N)=-p^{2 \sigma}$ for some integer $\sigma$ such that $1 \leq \sigma \leq 10$. The integer $\sigma$ is  called the \emph{Artin invariant} of the lattice $N$. If $X$ is a supersingular K3 surface, we call $\sigma(\mathrm{NS}(X))$ the \emph{Artin invariant} of the supersingular K3 surface $X$ and we find that $\sigma(\mathrm{NS}(X))=\sigma(H^2_{\mathrm{crys}}(X/W))$. The following theorem is due to Rudakov and Shafarevich \cite[Section 1]{MR633161}.
\begin{theorem} If $p \neq 2$, then the Artin invariant $\sigma$ determines a supersingular K3 lattice up to isometry.
\end{theorem}
\subsection{Characteristic subspaces and K3 crystals}
In this subsection we introduce characteristic subspaces. These objects yield another way to describe K3 crystals, a little closer to classic linear algebra in flavor. For this subsection we fix a prime $p > 2$ and a perfect field $k$ of characteristic $p$ with Frobenius $F \colon k \rightarrow k$, $x \mapsto x^p$.
\begin{definition} Let $\sigma$ be a non-negative integer and let $V$ be a $2\sigma$-dimensional $\mathbb{F}_p$-vector space. A non-degenerate quadratic form
\begin{align*}
\langle -,- \rangle \colon V \times V \rightarrow \mathbb{F}_p.
\end{align*}
on $V$ is called \emph{non-neutral} if there exists no $\sigma$-dimensional isotropic subspace of $V$.
\end{definition}
\begin{definition}
Let $\sigma$ be a non-negative integer and let $V$ be a $2\sigma$-dimensional $\mathbb{F}_p$-vector space together with a non-degenerate and non-neutral quadratic form
\begin{align*}
\langle -,- \rangle \colon V \times V \rightarrow \mathbb{F}_p.
\end{align*}
Set $\varphi \coloneqq \mathrm{id}_V \otimes F \colon V \otimes k \rightarrow V \otimes k$. A $k$-subspace $G \subset V \otimes k$ is called \emph{characteristic} if
\begin{enumerate}
\item $G$ is a totally isotropic subspace of dimension $\sigma$, and
\item $G + \varphi(G)$ is of dimension $\sigma+1$.
\end{enumerate}
A \emph{strictly characteristic subspace} is a characteristic subspace $G$ such that
\begin{align*}
V \otimes k = \sum_{i=0}^{\infty} \varphi^i(G)
\end{align*}
holds true.
\end{definition}
We can now introduce the categories
\begin{align*}
\mathrm{K}3(k) &\coloneqq \left\{\begin{array}{l}\text{Supersingular K3 crystals} \\ \text{with only isomorphisms as morphisms} \end{array} \right\}
\intertext{and}
\mathbb{C}3(k) &\coloneqq \left\{\begin{array}{l}\text{Pairs $(T,G)$, where $T$ is a supersingular} \\ \text{K3 lattice over $\mathbb{Z}_p$, and $G\subseteq T_0 \otimes_{\mathbb{Z}_p} k$} \\
\text{is a strictly characteristic subspace} \\
\text{with only isomorphisms as morphisms} \end{array} \right\}.
\end{align*}
It turns out that over an algebraically closed field these two categories are equivalent.
\begin{theorem}\cite[Theorem 3.20]{MR563467} Let $k$ be an algebraically closed field of characteristic $p > 0$. Then the functor
\begin{align*}
\mathrm{K}3(k) &\lra \mathbb{C}3(k), \\
\left(H, \varphi, \langle -,- \rangle\right) & \longmapsto \left(T_H, \ker\left(T_H \otimes_{\mathbb{Z}_p} k \rightarrow H \otimes_{\mathbb{Z}_p} k\right) \subset T_0 \otimes_{\mathbb{Z}_p} k\right)
\end{align*}
defines an equivalence of categories.
\end{theorem}
If we denote by $\mathbb{C}3(k)_{\sigma}$ the subcategory of $\mathbb{C}3(k)$ consisting of objects $(T,G)$ where $T$ is a supersingular K3 lattice of Artin invariant $\sigma$, then there is a coarse moduli space.
\begin{theorem}\cite[Theorem 3.21]{MR563467} Let $k$ be an algebraically closed field of characteristic $p > 0$. We denote by $\mu_n$ the cyclic group of $n$-th roots of unity. There exists a canonical bijection
\begin{align*}
\left(\mathbb{C}3(k)_{\sigma}/\simeq\right) \lra \mathbb{A}^{\sigma-1}_k(k)/\mu_{p^{\sigma}+1}(k).
\end{align*}

\end{theorem}
The previous theorem concerns characteristic subspaces defined on closed points with algebraically closed residue field. Next, we consider families of characteristic subspaces. 
\begin{definition} Let $\sigma$ be a non-negative integer and let $(V, \langle -,- \rangle)$ be a $2\sigma$-dimensional $\mathbb{F}_p$-vector space together with a non-neutral quadratic form. If $A$ is an $\mathbb{F}_p$-algebra, a direct summand $G \subset V \otimes_{\mathbb{F}_p} A$ is called a \emph{geneatrix} if $\mathrm{rk}(G)= \sigma$ and $\langle -,- \rangle$ vanishes when restricted to $G$. A \emph{characteristic geneatrix} is a geneatrix $G$ such that $G + F_A(G)$ is a direct summand of rank $\sigma+ 1$ in $V \otimes_{\mathbb{F}_p} A$. We write $\underline{M}_V(A)$ for the set of characteristic geneatrices in $V \otimes_{\mathbb{F}_p} A$.
\end{definition}
It turns out that there exists a moduli space for characteristic geneatrices.
\begin{proposition}\cite[Proposition 4.6]{MR563467} The functor
\begin{align*}
(\text{$\mathbb{F}_p$-algebras})^{\mathrm{op}} &\lra (\text{Sets}), \\
A &\longmapsto \underline{M}_V
\end{align*}
is representable by an $\mathbb{F}_p$-scheme $M_V$ which is smooth, projective and of dimension $\sigma-1$.
\end{proposition}
If $N$ is a supersingular K3 lattice with Artin invariant $\sigma$, then $N_0= pN^{\vee}/pN$ is a $2\sigma$-dimensional $\mathbb{F}_p$-vector space together with a non-degenerate and non-neutral quadratic form induced from the bilinear form on $N$.
\begin{definition} We set $\CM_{\sigma} \coloneqq M_{N_0}$ and call this scheme the \emph{moduli space of $N$-rigidified K3 crystals}. \end{definition}
\section{Enriques quotients of K3 surfaces of finite height}\label{lift}
Lieblich and Maulik showed in \cite{2011arXiv1102.3377L} that finite height K3 surfaces in positive characteristic admit well behaved lifts to characteristic zero, and we will use these lifting techniques and the fact that K3 surfaces over the complex numbers only have finitely many Enriques quotients to show that the same holds in positive characteristic. 

Let $k$ be an algebraically closed field and let $X$ be a K3 surface over $k$ with Néron-Severi lattice $\mathrm{NS}(X)$. We denote the \emph{group of isometries of $\mathrm{NS}(X)$} by $O(\mathrm{NS}(X))$. The \emph{positive cone $\CC_X$} is the connected component of $\{x \in \mathrm{NS}(X) \otimes \mathbb{R} \mid x^2 > 0 \} \subseteq \mathrm{NS}(X) \otimes \mathbb{R}$ that contains an ample divisor. The \emph{ample cone $\CA_X$} is the subcone of $\CC_X$ generated as a semigroup by ample divisors multiplied by positive real numbers. The set $\Delta_{\mathrm{NS}(X)} \coloneqq \{l \in \mathrm{NS}(X) \mid l^2=-2\}$ is called the set of \emph{roots of $\mathrm{NS}(X)$}. The \emph{Weyl group $W_{\mathrm{NS}(X)} = W_X$} of $\mathrm{NS}(X)$ is the subgroup of the orthogonal group $O(\mathrm{NS}(X))$ generated by all automorphisms of the form $s_l \colon x \mapsto x + \langle x,l \rangle l$ with $l \in \Delta_{\mathrm{NS}(X)}$. We set 
\begin{align*} O^{+}(\mathrm{NS}(X)) &\coloneqq \{ \varphi \in O(\mathrm{NS}(X)) \mid \varphi(\CA_X) = \CA_X \} 
 \intertext{to be the group of isometries of $\mathrm{NS}(X)$ that preserve the ample cone. Further, we define}
O_0(\mathrm{NS}(X)) &\coloneqq \ker\left(O(\mathrm{NS}(X)) \rightarrow O(q_{\mathrm{NS}(X)})\right)
\intertext{and} O_0(\mathrm{NS}(X))^{+} &\coloneqq O_0(\mathrm{NS}(X)) \cap O^{+}(\mathrm{NS}(X)). \end{align*}
We will need the following easy lemma.
\begin{proposition} \label{conj} Let $X$ be a K3 surface over an arbitrary field $k$ and let $\iota_1$ and $\iota_2$ be fixed point free involutions on $X$. Then the Enriques surfaces $X/\iota_1$ and $X/\iota_2$ are isomorphic if and only if there exists some automorphism $g \in \mathrm{Aut}(X)$ such that $g\iota_1g^{-1}=\iota_2$. \end{proposition}
\begin{proof} This is \cite[Proposition 2.1.]{MR2319542}. The proof does not depend on the base field. \end{proof}
\begin{theorem} \label{finht} Let $X$ be a K3 surface over an algebraically closed field $k$. If $X$ is of finite height, then the number of isomorphism classes of Enriques quotients of $X$ is finite. \end{theorem}
\begin{proof} Note that \cite[Theorem 6.1(1)]{2011arXiv1102.3377L} also holds for K3 surfaces of finite height over a field of characteristic $p=2$. Then \cite[Theorem 4.3]{MR1278263} together with \cite[Lemma 1.4(a),(c)]{MR2319542} and Proposition \ref{conj} implies the result. \end{proof}

\begin{remark}\label{near0} The theory of Enriques quotients of K3 surfaces of finite height is closely related to the characteristic zero situation. In many cases, given a finite height K3 surface $X$, we can choose a Neron-Severi group preserving lift $\CX_1$ of $X$ such that the specialization morphism $\gamma \colon \Aut(\CX_1) \lra \Aut(X)$ is an isomorphism. 

This is possible, for example, when $X$ is \emph{ordinary}, that means $X$ is of height $1$ \cite{MR723215} \cite[Theorem 4.11]{MR3992316} \cite[Proposition 2.3]{laface2019ordinary}.  Another class for which such lifts exist are the so-called \emph{weakly tame} K3 surfaces over fields of characteristic $p \geq 3$. In particular, every K3 surface of finite height over a field $k$ of characteristic $p \geq 23$ is weakly tame. For definitions and details we refer to \cite{MR3658183}.

In these situations we can then use the results from \cite{MR2319542} to obtain the number of isomorphism classes of Enriques quotients of $X$.   \end{remark}
\section{The supersingular case}\label{ssenq}
Let $X$ be a supersingular K3 surface over an algebraically closed field $k$ of characteristic $p \geq 3$. The following proposition shows that $X$ only has finitely many isomorphism classes of Enriques quotients.
\begin{proposition} Let $X$ be a supersingular K3 surface over an algebraically closed field $k$ of characteristic $p \geq 3$. The number of isomorphism classes of Enriques quotients of $X$ is finite. \end{proposition}
\begin{proof} We can use the same argument as in the proof of Theorem \ref{finht}. \end{proof}

\begin{remark}
Over characteristic $p=2$ the previous result does not hold. Indeed, the supersingular K3 surface of Artin invariant $1$ over a field of characteristic $2$ has infinitely many Enriques quotients \cite{katsura20141dimensional}.
\end{remark}

Our goal for the rest of this section is to find a formula for the number of Enriques quotients of $X$ in the style of \cite[Theorem 2.3]{MR2319542}. The argument does not rely on the previous proposition.

If $Y$ is an Enriques surface, then the torsion free part of its Neron-Severi group $\mathrm{NS}(Y)$ is isomorphic to the lattice $\Gamma =  U_2 \oplus E_8(-1)$, which is up to isomorphism the unique unimodular, even lattice of signature $(1,9)$. Following \cite{MR2319542}, if $X$ is a supersingular K3 surface over a field of characteristic $p \geq 3$, we define
\begin{align*} \mathfrak{M} &\coloneqq \left\{\begin{array}{l|l}
N \subseteq \mathrm{NS}(X)  & \text{primitive sublattices satisfying} \\
 & (A)\colon N \cong \Gamma(2)  \\
 & (B) \colon \text{No vector of square } -2 \text{ in } \mathrm{NS}(X) \text{ is orthogonal to } N 
\end{array} \right\}  \intertext{and}
\mathfrak{M}^{\ast} &\coloneqq \left\{N \in \mathfrak{M} \mid \text{$N$ contains an ample divisor}\right\}.\end{align*}
The following proposition describes free involutions on a supersingular K3 surface in terms of embeddings of lattices.
\begin{proposition}\cite[Theorem 4.1]{2013arXiv1301.1118J} \label{bij} Let $k$ be a an algebraically closed field of characteristic $p \geq 3$. For a supersingular K3 surface $X$ over $k$, there is a natural bijection
\begin{align*} \mathfrak{M}^{\ast} \stackrel{1 : 1}\longleftrightarrow \{\text{free involutions of } X\}. \end{align*}\end{proposition}
\begin{proof}[Idea of proof] For the convenience of the reader, we briefly recall the idea of the proof.

First, let $\iota \colon X \rightarrow X$ be a free involution of $X$ and let $f\colon X \rightarrow Y$ be the associated Enriques quotient. Since the map $f$ is finite étale of degree $2$, we obtain a primitive embedding of lattices 
\begin{align*}U(2) \oplus E_8(2) \cong f^{\ast}(\mathrm{NS}(Y)) \hookrightarrow \mathrm{NS}(X). \end{align*}
We write $N=f^{\ast}(\mathrm{NS}(Y))$ and $M = N^{\perp}$, such that 
\begin{align*} 
N= \{ v \in \mathrm{NS}(X)\mid \iota^{\ast}(v)=v\} \text{ and } M=\{ v \in \mathrm{NS}(X)\mid \iota^{\ast}(v)=-v\}.   \end{align*}
Then $N$ has property $(B)$: By the Riemann-Roch theorem, if $v$ is a $(-2)$-divisor on $X$, then $v$ or $-v$ is effective. Thus, if $v \in M$ was a $(-2)$-divisor, then both $v$ and $-v$ are effective, which is absurd. Pullback along finite morphisms preserves ampleness, hence $N$ contains an ample line bundle and we have shown that $N \in \mathfrak{M}^{\ast}$. 

On the other hand, assume we are given some $N \in \mathfrak{M}^{\ast}$ and define
\begin{align*}\psi \colon N \oplus N^{\perp} &\lra N \oplus N^{\perp}, \\
(v,w) &\longmapsto (v, -w). \end{align*}
Then $\psi$ extends to $\mathrm{NS}(X)$ \cite[Lemma 4.2.]{2013arXiv1301.1118J} and by the supersingular Torelli theorem \cite{MR717616} induces an involution $\iota$ on $X$. From condition (B) it then follows that $\iota$ indeed has no fixed points. \end{proof}
If $A$ is a finitely generated abelian group and $q$ a prime number, we denote by $A^{(q)}$ the $q$-torsion part of $A$ and by $l(A)$ the minimal cardinality among all sets of generators of $A$.
\begin{lemma}\label{surNS} Let $k$ be an algebraically closed field of characteristic $p \geq 3$ and $X$ a supersingular K3 surface over $k$. The canonical morphism $\mathrm{pr} \colon O(\mathrm{NS}(X)) \rightarrow O(q_{\mathrm{NS}(X)})$ is surjective. \end{lemma}
\begin{proof} The Néron-Severi lattice of a supersingular K3 surface $X$ is even, indefinite and non-degenerate with $\mathrm{rk}(\mathrm{NS}(X))=22$ and $2 \leq l(A_{\mathrm{NS}(X)}^{(p)}) \leq 20$, $l(A_{\mathrm{NS}(X)}^{(q)})=0$ for any prime $q \neq p$. Now the lemma follows from \cite[Theorem 1.14.2]{1980IzMat..14..103N}. \end{proof}
Let $k$ be a perfect field in positive characteristic $p > 0$ and let $W(k)$ be the Witt ring over $k$, then we denote by $\mathrm{Cart}(k)$ the non-commutative ring $W(k)\langle \langle V \rangle \rangle \langle F \rangle$ of power series in $V$ and polynomials in $F$ modulo the relations
$$FV=p\text{, }VrF=V(r)\text{, }Fr=\sigma(r)F\text{, }rV=V\sigma(r)\text{ for all } r \in W(k),$$
where $\sigma(r)$ denotes Frobenius of $W(k)$ and $V(r)$ denotes Verschiebung of $W(k)$.

We will need the following lemma and proposition. The statement we need to show has already been proved in \cite[Theorem 2.1 and Remark 2.2]{Nygaard1980}, but not been stated explicitly.  We will therefore give a full proof. 

When $G$ is a formal group law, we write $DG$ for the associated Dieudonné module as in \cite[Section 1]{MR0257089}.
\begin{lemma} \label{dieud} Let 
\begin{align*} \psi \colon D\hat{\mathbb{G}}_a \stackrel{\cong}\lra D\hat{\mathbb{G}}_a \end{align*}
be a continuous automorphism of left $\mathrm{Cart}(k)$-modules such that there exists a non-trivial finite dimensional $k$-subvector space $U \subset D\hat{\mathbb{G}}_a$ with $\psi(U) \subseteq U$. Then $\psi$ is the multiplication by some element $a \in k^{\times}$ from the right. \end{lemma}
\begin{proof} We have
\begin{align*} D\hat{\mathbb{G}}_a = \prod_{i=0}^{\infty} V^i k \end{align*}
as a $\mathrm{Cart}(k)$-module with trivial $F$-action and $W$-action coming from the projection $W \twoheadrightarrow k$. 

We let $ \psi \colon D\hat{\mathbb{G}}_a \stackrel{\cong}\lra D\hat{\mathbb{G}}_a$ be an automorphism such that $\psi(1)= \sum_{i=0}^{\infty} a_i V^i$ and take an arbitrary element $x= \sum x_j V^j \in D\hat{\mathbb{G}}_a$. Then, since $Va_i= a_i^{\frac{1}{p}}$ it follows by continuity that 
\begin{align*} \psi(x) = \sum a_i^{\frac{1}{p^j}} x_j V^{i+j}. \end{align*}
In other words, we can regard $\psi$ as the $k$-linear automorphism of $k[\![V]\!]$ given by multiplication with $a = \sum a_i V^i \in k[\![V]\!]$ from the right. We want to see that $a$ is an element of $k$. 

Since $\psi$ is an automorphism, we have that $a_0 \neq 0$. When $x = \sum_{i=0}^{\infty} b_i V^i \in  D\hat{\mathbb{G}}_a$ is a power series, we write $\mathrm{subdeg}(x)= \min \{i \mid b_i \neq 0 \}$. We assume that $a \notin k^{\times}$ and let $u^{(0)} \in U$. Then 
\begin{align*} u^{(1)} \coloneqq \psi\left(u^{(0)}\right) - a_0^{\left(p^{-\mathrm{subdeg}\left(u^{\left(0\right)}\right)}\right)}u^{(0)} \intertext{is also an element of $U$ and we have $\mathrm{subdeg}(u^{(1)}) > \mathrm{subdeg}(u^{(0)})$. Inductively, taking}
u^{(n+1)} \coloneqq \psi\left(u^{(n)}\right)-a_0^{\left(p^{-\mathrm{subdeg}\left(u^{\left(n\right)}\right)}\right)}u^{\left(n\right)}, \end{align*}
 we find that $u^{(n+1)}$ is an element of $U$ with $\mathrm{subdeg}(u^{(n+1)}) > \mathrm{subdeg}(u^{(n)})$. This is a contradiction to the finiteness of the dimension of $U$ and hence concludes the proof of the lemma.     \end{proof}
With the use of the technical Lemma \ref{dieud} we can prove the following nice observation. 
\begin{proposition} \label{nice} Let $k$ be an algebraically closed field of characteristic $p \geq 3$ and let $X$ be a supersingular K3 surface of Artin invariant $\sigma_X$ over $k$ such that the point corresponding to $X$ in the moduli space of supersingular K3 crystals $\mathbb{A}_k^{\sigma_X-1}/\mu_{p^{\sigma_X}+1}$ has coordinates $(b_1, \ldots, b_{\sigma_X-1})$ with $b_1 \neq 0$. Let $\theta \in \mathrm{Aut}(X)$ be an automorphism of $X$. Then the induced automorphism $\theta^{\ast} \in O(q_{\mathrm{NS}(X)})$ of $A_{\mathrm{NS}(X)}$ is the identity or multiplication with $-1$. \end{proposition}
\begin{proof} To simplify notation, we write $\mathrm{NS}=\mathrm{NS}(X)$ and $\sigma= \sigma_X$. Since there exists a natural isomorphism of lattices $A_{\mathrm{NS}} \otimes k \cong T_0 \otimes k$, it follows from \cite[Theorem 1.12]{Nygaard1980} that there exists a functorial embedding $A_{\mathrm{NS}} \otimes k \hookrightarrow H^2(X, W\CO_X)$.

More precisely, from \cite[Lemma 1.11]{Nygaard1980} it follows that the image of the quadratic space $A_{\mathrm{NS}} \otimes k$ in $H^2(X, W\CO_X) \cong D\Phi^2_X = k [\![V]\!]$ has basis $\{1, \ldots, V^{2\sigma - 1}\}$. Further, the embedding 
\begin{align*} H^2_{\mathrm{cris}} (X/W) / (\mathrm{NS} \otimes W) \hookrightarrow H^2(X, W\CO_X) \end{align*}
identifies $H^2_{\mathrm{cris}} (X/W) / (\mathrm{NS} \otimes W)$ with the subspace of $A_{\mathrm{NS}} \otimes k$ with basis $\{1, \ldots, V^{\sigma-1}\}$ and it follows from \cite[Proposition 2.12]{MR717616} that this is a strictly characteristic subspace.

We write $\langle -,- \rangle$ for the bilinear form on $A_{\mathrm{NS}} \otimes k$ and we claim that $\langle V^{\sigma-1}, V^{2\sigma-1} \rangle \neq 0$. Indeed, we have that $\mathrm{span}(1, \ldots, V^{\sigma-1})$ is a maximal isotropic subspace in $A_{\mathrm{NS}} \otimes k$. We assume that we have $\langle V^{\sigma -1}, V^{2\sigma-1} \rangle = 0$. We write $\varphi \colon A_{\mathrm{NS}} \otimes k \rightarrow A_{\mathrm{NS}} \otimes k$ for the action of the Frobenius. For $1 < n \leq \sigma$ we find 
\begin{align*} \langle V^{\sigma - n}, V^{2\sigma-1} \rangle = \langle \varphi^{1-n}(V^{\sigma -1}), \varphi^{1-n}( V^{n-2}) \rangle = \langle V^{\sigma-1}, V^{n-2} \rangle = 0. \end{align*}
Thus, the space $\mathrm{span}(1, \ldots, V^{\sigma-1})+ \langle V^{2\sigma-1} \rangle$ would be isotropic. This yields a contradiction.

Now let $\theta \colon X \rightarrow X$ be an automorphism. Then the induced $\theta^{\ast} \colon A_{\mathrm{NS}} \otimes k \rightarrow A_{\mathrm{NS}} \otimes k$ is an automorphism of quadratic spaces and it follows from Lemma \ref{dieud} that $\theta^{\ast}(V^i) = a^{\frac{1}{p^i}}$ for some $a \in k^{\times}$ and all $i \in \mathbb{N}$. Thus, we find
\begin{align*}
\langle V^{\sigma - 1}, V^{2\sigma-1} \rangle &= \langle \theta^{\ast}(V^{\sigma - 1}), \theta^{\ast}(V^{2\sigma}-1) \rangle \\
&= \langle a^{\frac{1}{p^{\sigma - 1}}} V^{\sigma - 1},a^{\frac{1}{p^{2\sigma-1}}} V^{2\sigma-1} \rangle \\
&= a^{\frac{1+p^{\sigma}}{p^{2\sigma-1}}} \langle V^{\sigma - 1}, V^{2\sigma-1} \rangle \end{align*}
and it follows that $a^{p^{\sigma}+1}=1$. 

On the other hand, from \cite[Proposition 1.18]{Nygaard1980} we get that
\begin{align*}
b_1 &= \langle V^{\sigma-2}, V^{2\sigma-1} \rangle.
\intertext{Since $b_1 \neq 0$, it follows from}
b_1 &= \langle V^{\sigma-2}, V^{2\sigma-1} \rangle \\
  &= \langle \theta^{\ast}(V^{\sigma-2}), \theta^{\ast}(V^{2\sigma-1}) \rangle \\
  &= a^{\frac{p^{\sigma+1}+1}{p^{2\sigma-1}}}\langle V^{\sigma-2}, V^{2\sigma-1} \rangle
\intertext{that we have $a^{p^{\sigma+1}+1}=1$. Thus, we find}
1 &= \frac{a^{p^{\sigma+1}+1}}{a^{p^{\sigma}+1}}= (a^{p-1})^{p^{\sigma}}
\intertext{and therefore also}
1 &= a^{p-1}. \end{align*}
In other words, we have that $a \in \mathbb{F}_p$. But then the morphism $\theta^{\ast} \colon A_{\mathrm{NS}} \otimes k \rightarrow A_{\mathrm{NS}} \otimes k$ is just multiplication by $a$ and from the equality $a^{p^{\sigma}+1}=1$ it follows that $a^2=1$. \end{proof}
\begin{remark}
An alternative proof of Proposition \ref{nice} can be found in \cite[Theorem 5.11, Lemma 5.15]{brandthesis}
\end{remark}
\begin{remark}\label{dense} Of course, the subset of $\mathbb{A}_k^{\sigma_X-1}/\mu_{p^{\sigma_X}+1}$ consisting of points $(b_1, \ldots, b_{\sigma_X-1})$ with $b_1 \neq 0$ is open. If $\sigma > 1$, then this subset is also dense in $\mathbb{A}_k^{\sigma_X-1}/\mu_{p^{\sigma_X}+1}$. It follows from \cite[Proposition 4.10]{MR563467} that in this case the corresponding subset in the period space of supersingular K3 surfaces $\CM_{\sigma}$ is also dense. \end{remark}
\begin{remark} There are also supersingular K3 surfaces $X$ with $b_1=0$ such that each automorphism of $X$ induces either the identity or multiplication by $-1$ on the transcendental lattice. For example, let $X$ be with $\sigma_X=4$ and such that $b_1=0$ and $b_2=1$. Going back to the argument in the proof of Proposition \ref{nice} we then find
\begin{align*}
1= \langle V^{\sigma - 3}, V^{2\sigma -1} \rangle = a^{\frac{p^{\sigma+2}+1}{p^{2\sigma-1}}}\langle V^{\sigma - 3}, V^{2\sigma -1} \rangle,
\end{align*}
and it thus follows that $a^{p^{\sigma+2}+1}=1=a^{p^{\sigma}+1}$. Hence, it is $(a^{p^2-1})^{p^{\sigma}}=1$ and we find  $a \in \mathbb{F}_{p^2}$. But then, using that $\sigma=4$, we have
\begin{align*}
1 = a^{{p^4}+1} = (a^{p^2})^{p^2} \cdot a = a^2
\end{align*}
and we can conclude as in the proof of Proposition \ref{nice}. \end{remark}
\begin{remark} On the other hand, there also exist examples of supersingular K3 surfaces $X$ and automorphisms $\theta \in \mathrm{Aut}(X)$ such that the induced morphism on $\mathrm{NS}(X)^{\vee}/\mathrm{NS}(X)$ is not the identity or multiplication by $-1$. For example if $\sigma_X=1$, then the image of the canonical map $\mathrm{Aut}(X) \rightarrow \mathrm{Aut}(\mathrm{NS}(X)^{\vee}/\mathrm{NS}(X))$ is known to be a cyclic group of order $p+1$ \cite[Remark 3.4]{MR3466820}.
\end{remark}
The following theorem is the supersingular version of a characteristic zero theorem by Ohashi \cite[Theorem 2.3.]{MR2319542}. Similar to the situation in characteristic zero we only obtain an inequality in general. In characteristic zero there are two conditions on a K3 surface $X$ that have to be fullfilled in order to obtain an equality. One of these is the surjectivity of the canonical morphism $\mathrm{pr} \colon O(\mathrm{NS}(X)) \rightarrow O(q_{\mathrm{NS}(X)})$. This is always true for supersingular K3 surfaces by Lemma \ref{surNS}. The other condition is that each automorphism of $X$ induces $\pm \mathrm{id}$ on the transcendental lattice of $X$. We gave a sufficient criterion under which this is always true in Proposition \ref{nice}.
\begin{theorem}\label{form} Let $k$ be an algebraically closed field of characteristic $p \geq 3$ and let $X$ be a supersingular K3 surface over $k$. Let $M_1, \ldots, M_n \in \mathfrak{M}$ be a complete set of representatives for the action of $O(\mathrm{NS}(X))$ on $\mathfrak{M}$. For each $j \in \{1, \ldots, n\}$, we let
\begin{align*} K^{(j)}=\{\psi \in O(\mathrm{NS}(X)) \mid \psi(M_j)=M_j \}\end{align*}
be the stabilizer of $M_j$ and $\mathrm{pr}(K^{(j)})$ be its canonical image in $O(q_{\mathrm{NS}(X)})$. Then we have inequalities
\begin{align*} n \leq \#\left\{\text{Enriques quotients of } X \right\} \leq \sum_{j=1}^n \#\left(O\left(q_{\mathrm{NS}(X)}\right)/\mathrm{pr}\left(K^{(j)}\right)\right). \end{align*}
If $X$ is such that for each automorphism $\theta \in \mathrm{Aut}(X)$ the induced automorphism on the quotient $\mathrm{NS}(X)^{\vee}/\mathrm{NS}(X)$ is either the identity or multiplication by $-1$, then the inequality above becomes an equality on the right side. \end{theorem}
\begin{proof} It follows from \cite[Proposition 1.15.1.]{1980IzMat..14..103N} that the number of representatives $M_j$ is indeed finite. Therefore, using Proposition \ref{bij} and Lemma \ref{surNS}, the proof goes word by word as the proof of \cite[Theorem 2.3.]{MR2319542}.  \end{proof}
\begin{remark} It follows from Remark \ref{dense} that for a general supersingular K3 surface $X$ of Artin invariant $\sigma > 1$ the inequality on the right hand side in Theorem \ref{form} is an equality.\end{remark}
\section{Existence of Enriques quotients for supersingular K3 surfaces}
In the previous section, in Theorem \ref{form} we gave a formula which computes the number of Enriques quotients for a general supersingular K3 surface $X$. However, it turns out that explicitly calculating this number is difficult. A priori it is not even clear that this number is non-zero, or in other words that for a given supersingular K3 surface $X$ the corresponding set of lattices $\mathfrak{M}$ is non-empty. The following result is due to J.\ Jang.
\begin{proposition}\cite[Theorem 4.5, Proposition 3.5]{2013arXiv1301.1118J} Let $k$ be an algebraically closed field of characteristic $p$ and let $X$ be a supersingular K3 surface of Artin invariant $\sigma$. If $\sigma=1$, then $X$ has an Enriques involution. If $\sigma \in \{3,5\}$, and $p=11$ or $p \geq 19$, then $X$ has an Enriques involution. If $\sigma \in \{2,4\}$, and $p = 19$ or $p \geq 29$, then $X$ has an Enriques involution. If $\sigma \geq 6$, then $X$ has no Enriques involution. \end{proposition}
The idea of the proof is as follows. Associated to a supersingular K3 surface $X$ of Artin invariant $\sigma$ over a field $k$ of characteristic $p$ one constructs a K3 surface $X_{\sigma,d}$ over $\mathbb{C}$ such that the transcendental lattice $T(X_{\sigma, d})$ is isomorphic to a lattice $U(2) \oplus M_{\sigma, d}$ where $M_{\sigma,d}$ is a certain lattice that admits an embedding into $\Gamma(2)$ such that its orthogonal complement does not contain any $(-2)$-vectors. For large enough characteristic $p$ as in the statement of the proposition one can choose $d$ such that we find a chain of primitive embeddings $\Gamma(2) \hookrightarrow \mathrm{NS}(X_{\sigma, d}) \hookrightarrow \mathrm{NS}(X)$. In this situation one can show that the orthogonal complement of $U(2) \oplus E_8(2)$ in $\mathrm{NS}(X)$ does not contain any $(-2)$-vectors. However, this method is not applicable for small $p$. We note that there are only $24$ cases left to work out and we can try to show the existence of an Enriques quotient in those remaining cases by hand.
\begin{theorem}\label{exist} Let $k$ be an algebraically closed field of characteristic $p$ where $p \geq 3$ and let $X$ be a supersingular K3 surface of Artin invariant $\sigma$. Then $X$ has an Enriques involution if and only if $\sigma \leq 5$. \end{theorem}
This result has already been shown by Jang in a later paper \cite{MR3350105} via lifting techniques, but we want to reprove it using the lattice argument which we described above. 
\subsection{Computational approach}
Let $X$ be a supersingular K3 surface of Artin invariant $\sigma$ over an algebraically closed field $k$ with characteristic $p \geq 3$. By the results in the previous section, it suffices to show that there exists a primitive embedding of the lattice $\Gamma(2)$ into $\mathrm{NS}(X)$ such that the orthogonal complement of $\Gamma(2)$ in $\mathrm{NS}(X)$ does not contain any vector of self-intersection $-2$. We denote by $A_{S_{p,\sigma}}$ the discriminant group of $\mathrm{NS}(X)$ and by $q_{S_{p,\sigma}}$ the quadratic form on $A_{S_{p,\sigma}}$, similarly we write $A_{\Gamma(2)}$ for the discriminant group of $\Gamma(2)$ and $q_{\Gamma(2)}$ for the quadratic form on $A_{\Gamma(2)}$. 

\begin{remark} The lattice $\mathrm{NS}(X)$ is the unique lattice up to isomorphism in its genus \cite[Section 1]{MR633161}, so by \cite[Proposition 1.15.1]{1980IzMat..14..103N} the datum of a primitive embedding $\Gamma(2) \hookrightarrow \mathrm{NS}(X)$ with orthogonal complement $L$ is equivalent to the datum of an even lattice $L$ with invariants $(0,12, \delta_{p,\sigma})$ where $\delta_{p,\sigma}$ is the form $-q_{S_{p,\sigma}} \oplus q_{\Gamma(2)}$ with domain $A_{S_{p,\sigma}} \oplus A_{\Gamma(2)}$ and $(0,12)$ is the signature of $L$. To see this, observe that in our case $\# A_{S_{p,\sigma}}=p^{2\sigma}$ and $\# A_{\Gamma(2)}=2^{10}$ are coprime, and so the isomorphism of subgroups $\gamma$ in the cited proposition has to be the zero-morphism. \end{remark}

It follows from the previous remark, that to prove Theorem \ref{exist}, we have to construct lattices $L_{p,\sigma}$ of genus $(0,12,\delta_{p,\sigma})$ such that the $L_{p, \sigma}$ do not contain any vectors of self-intersection $-2$. Using the computer algebra program \textsc{magma} we constructed the lattices $L_{p,\sigma}$ in the missing cases. I am indebted to Markus Kirschmer for helping me with using the program and writing code to automatize step 1 of the following method:
\begin{itemize}
\item Step 1. Construct an arbitrary lattice $L$ of genus $(0,12,\delta_{p, \sigma})$. This can be done for example in the following way. Using \cite[Chapter 1.]{MR633161} we can construct the lattice $\mathrm{NS}(X)$ explicitly. Then we choose an arbitrary primitive embedding $N \hookrightarrow \mathrm{NS}(X)$ and take $L$ to be the orthogonal complement under this embedding. We remark that in general the lattice $L$ may contain vectors of self intersection $-2$.
\item Step 2. Apply Kneser's neighbor method \cite{MR0090606}, which has been implemented for \textsc{magma}, to the positive definite lattice $-L$. This generates a list of further candidate lattices in the same genus as $-L$. Using the "Minimum()" function in \textsc{magma} we can test for the minimum length of vectors in those candidate lattices until we find a candidate that does not contain any vectors of length $2$.
\end{itemize} Note that we might have to iterate the neighbor method. 

Applying the above method, we found a list of lattices $L_{p,q}$ of genus $(0,12,\delta_{p,q})$ that do not contain any vectors of self intersection $-2$. We represent these lattices via their Gram matrix and these Gram matrices can be found in the attached .txt-file. Their existence in conjuction with the results from \cite{2013arXiv1301.1118J} imply Theorem \ref{exist}.

\begin{remark} In theory, with the presented approach, it should be possible to explicitly compute the general number of isomorphism classes of Enriques quotients of a supersingular K3 surface $X$ with given characteristic $p$ of the ground field $k$ and Artin invariant $\sigma$.

Namely, in Theorem \ref{form} the $M_i$ are members of isometry-classes of lattices in the genus $(0,12,\delta_{p,\sigma})$ that contain no $(-2)$-vectors. Two different isometry-classes in particular yield two different orbits for the action of $O(\mathrm{NS}(X))$. 

The \textsc{magma}-command \textit{Representatives(G);} computes a representative for every isometry-class in a given genus $G$. We can then distinguish the isometry-classes that contain no $(-2)$-vectors and compute the orthogonal group of their discriminant group as well as their stabilizer in $O(\mathrm{NS})$. We note that each of those steps still is very complicated. \end{remark}

\subsection{Lower bounds}
Using the method from the previous remark, we computed the number $\mathrm{Rep}(p,\sigma)$ of isometry-classes of lattices without $(-2)$-vectors for some genera $(0,12,\delta_{p,\sigma})$ in small characteristics. This yields a lower bound for the number of Enriques involutions of a supersingular K3 surface in these cases. However, since the groups $O(q_{\mathrm{NS}})$ are large already in these cases, this bound is possibly not optimal. We also note, that already in these comparatively simple cases, computing each of those numbers was very memory intensive.  

\begin{proposition}\label{lowb} For the number of isomorphism classes of Enriques quotients of a supersingular K3 surface of Artin invariant $\sigma$ over an algebraically closed ground field $k$ of characteristic $p$ we found the following numbers of equivalence classes under the action of $O(\mathrm{NS}(X))$ on $\mathfrak{M}$ denoted by $\mathrm{Rep}(p, \sigma)$: 
\begin{table}[H]
\caption{Some results for the lower bounds $\mathrm{Rep}(p, \sigma)$}
\centering 
\begin{tabular}{c c c c c c}
\hline
\hline                        
$p$ & $\sigma=1$ & $\sigma=2$ & $\sigma=3$ & $\sigma=4$ & $\sigma=5$ \\
 [0.5ex]
\hline
$3$ & $2$ & $12$ & $30$ & $20$ & $7$ \\
$5$ & $10$ & $222$ & $875$ & $302$ & $24$  \\
$7$ & $42$ & $3565$  & $?$ & $4313$ & $81$ \\
$11$ & $256$ & $?$ & $?$ & $?$ & $438$ \\
$13$ & $537$ & $?$ & $?$ & $?$ & $866$\\
$17$ & $2298$ & $?$ & $?$ & $?$ & $2974$ \\
 [1ex]
\hline
\end{tabular}
\label{table:nonlin}
\end{table}
\end{proposition}
\subsection{Upper bounds} 
In this section, we compute the cardinality of the quotients $O\left(q_{\mathrm{NS}(X)}\right)/\mathrm{pr}(K^{(j)})$ in Theorem \ref{form} in some cases. Therefore, we can use Theorem \ref{form} and Proposition \ref{lowb} to find the generic number of isomorphism classes of Enriques quotients for small $p$ and $\sigma$.
\begin{proposition} Let $X$ be a supersingular K3 surface and let $M \in \mathfrak{M}$ be a primitive sublattice of $\mathrm{NS}(X)$. If $\psi' \colon M^{\perp} \rightarrow M^{\perp}$ is an isometry of $M^{\perp}$, then there exists an isometry $\psi \colon \mathrm{NS}(X) \rightarrow \mathrm{NS}(X)$ of $\mathrm{NS}(X)$ such that $\psi|_{M^{\perp}}= \psi'$. In particular, we have $\psi(M)=M$. Further, the image of $\psi$ in $O\left(q_{\mathrm{NS}(X)}\right)$ only depends on $\psi'$.   \end{proposition}
\begin{proof} It follows from \cite[Theorem 1.14.2]{1980IzMat..14..103N} that the canonical morphism of orthogonal groups $O\left(\Gamma(2)\right) \rightarrow O\left(q_{\Gamma(2)}\right)$ is surjective. Since $M$ is isomorphic to $\Gamma(2)$ it thus follows from \cite[Corollary 1.5.2]{1980IzMat..14..103N} that for any automorphism $\psi' \colon M^{\perp} \rightarrow M^{\perp}$ we can choose an automorphism $\varphi' \colon M \rightarrow M$ such that $\psi' \oplus \varphi'$ extends to an automorphism $\psi$ of $\mathrm{NS}(X)$.

Since we have natural maps
\begin{align*}
\{ \psi \in O\left(\mathrm{NS}(X)\right) \mid \psi(M)=M \} \rightarrow O(M^{\perp}) \rightarrow O(q_{M^{\perp}}) \cong O\left(q_{M}\right) \oplus O\left(q_{\mathrm{NS}(X)}\right) \rightarrow O\left(q_{\mathrm{NS}(X)}\right)
\end{align*}
the second statement of the proposition follows. \end{proof}
In other words, in Theorem \ref{form} the subgroup $\mathrm{pr}\left(K^{(j)}\right)$ of $O\left(q_{\mathrm{NS}(X)}\right)$ is the image of $O\left(M_j^{\perp} \right)$ in $O\left(q_{\mathrm{NS}(X)}\right)$. Further, we have a natural isomorphism $\left(q_{M_j^{\perp}}\right)_p \cong O(q_{\mathrm{NS}(X)})$. We thus have the following corollary.
\begin{corollary}\label{alt} Let $k$ be an algebraically closed field of characteristic $p \geq 3$ and let $X$ be a supersingular K3 surface over $k$. Let $M_1, \ldots, M_n \in \mathfrak{M}$ be a complete set of representatives for the action of $O(\mathrm{NS}(X))$ on $\mathfrak{M}$. For each $j \in \{1, \ldots, n\}$, we write $\mathrm{im}\left(O(M^{\perp}_j)\right)$ for the image of $O(M^{\perp}_j)$ in $O\left(\left(q_{M_j^{\perp}}\right)_p\right)$ under the natural map $O(M_j^{\perp}) \rightarrow O(q_{M^{\perp}}) \rightarrow O\left(\left(q_{M_j^{\perp}}\right)_p\right)$. Then we have inequalities
\begin{align*} n \leq \#\left\{\text{Enriques quotients of } X \right\} \leq \sum_{j=1}^n \#\left(O\left(\left(q_{M_j^{\perp}}\right)_p\right)/\mathrm{im}\left(O(M^{\perp}_j)\right)\right). \end{align*}
If $X$ is such that for each automorphism $\theta \in \mathrm{Aut}(X)$ the induced automorphism on the quotient $\mathrm{NS}(X)^{\vee}/\mathrm{NS}(X)$ is either the identity or multiplication by $-1$, then the inequality above becomes an equality on the right side. \end{corollary}
We use these results to prove the following proposition.
\begin{proposition}\label{upb}
For the number of isomorphism classes of Enriques quotients of a supersingular K3 surface of Artin invariant $\sigma$ over an algebraically closed ground field $k$ of characteristic $p$ we found the following upper bounds. When $\sigma > 2$, then these are the numbers of isomorphism classes of Enriques quotients of a general supersingular K3 surface of Artin invariant $\sigma$.
\begin{table}[H]
\caption{Some results for the upper bounds}
\centering 
\begin{tabular}{c c c c c }
\hline
\hline                        
$p$ & $\sigma=1$ & $\sigma=2$ & $\sigma=3$ & $\sigma=4$ \\
 [0.5ex]
\hline
$3$ & $2$ & $490$ & $1278585$ & $24325222428$ \\
$5$ & $33$ & $635765$ & $1614527971875$ & $37184780652626927616$  \\
$7$ & $175$ & $191470125$  & $?$ & $88339146755283817573908480$  \\
$11$ & $2130$ & $?$ & $?$ & $?$ \\
$13$ & $5985$ & $?$ & $?$ & $?$ \\
$17$ & $36000$ & $?$ & $?$ & $?$  \\
 [1ex]
\hline
\end{tabular}
\end{table}

\begin{table}[H]
\caption{Some results for the upper bounds}
\centering 
\begin{tabular}{c c}
\hline
\hline                        
$p$ & $\sigma=5$ \\
 [0.5ex]
\hline
$3$  & $1286212218643287$ \\
$5$  & $1300418157436546004702724096
$  \\
$7$ &  $146385612443146033546182607153135616$ \\
$11$ & $9360899237983480445308665427637667976947171328$ \\
$13$ & $86881471802459725997082069598436845809167673327616$\\
$17$ & $117559509833496435964143968964217511931559374134458712064$ \\
 [1ex]
\hline
\end{tabular}
\end{table}

\end{proposition}
\begin{proof} Using the formula (2.4) for quadratic forms of type IV from \cite{MR172935} we can directly compute the cardinality of $O\left(q_{\mathrm{NS}(X)}\right)$ for a supersingular K3 surface $X$. From Corollary \ref{alt} it follows that then we only have to compute the cardinality of the image of $O(M_j^{\perp})$ in $O\left(\left(q_{M_j^{\perp}}\right)_p\right)$ for each $M_j \in \mathfrak{M}$. We did this with the computer algebra program SAGE. The following code was used for $p=7$ and $\sigma=4$.
\begin{python}
from multiprocess import Pool

B = [list of matrices]
def cnt(mat):
	L=IntegralLattice(mat)
	C=L.dual_lattice()
	O=L.orthogonal_group()
	Q=C/(L+7*C)
	Y, X= Q.optimized()
	Mat=matrix(Y.V().basis())
	R=Mat.rows()
	temp=[]
	for l in range(12):
		if (Mat.rref()).column(l).list().count(0) != 7 or (Mat.rref()).column(l).list().count(1) != 1:
			temp.append(l)

	for l in temp:
		R[0:0]=[identity_matrix(12).row(l)]

	R=matrix(R)
	R2= R.inverse()
	genim=[]
	for A in O.gens():
		t=R*A*R2
		t=t.transpose()
		t=t[4:12]
		t=t.transpose()
		t=t.rows()
		del t[0:4]
		t=matrix(GF(7),t)
		genim.append(t)
	G=MatrixGroup(genim)
	return [O.order(), G.order()]

if __name__ == '__main__':
    with Pool(4) as p:
        print(p.map(cnt, B))

\end{python}
We remark that there are alternative ways to compute the number we are interested in implemented in SAGE, however the way presented above was - among all the methods we tried - the most memory and CPU efficient.
\end{proof}
\subsection{The case $p=3$ and $\sigma=1$}
The situation where $p=3$ and $\sigma=1$ is particularly easy.
\begin{theorem}\label{p3s1} There are exactly two isomorphism classes of Enriques quotients of the supersingular K3 surface $X$ of Artin invariant $1$ over an algebraically closed field $k$ of \mbox{characteristic $3$}.
\end{theorem}
\begin{proof} Since we computed $\mathrm{Rep}(3,1)=2$, there are at least two isomorphism classes of Enriques quotients of $X$. On the other hand, it follows from Proposition \ref{upb} that there are at most two isomorphism classes of Enriques quotients of $X$ and we are done.
\end{proof}
In \cite{MR4009175}, Enriques surfaces with finite automorphism groups are classified and fall into seven types. We thank Gebhard Martin for communicating the following result to us.
\begin{proposition} Let $k$ be an algebraically closed field of characteristic $p=3$ and let $Y$ be the unique Enriques surface with finite automorphism group of type III (respectively of type IV) over $k$, following the classification in \cite{MR4009175}. Then, the K3-cover of $Y$ is the supersingular K3 surface $X$ with Artin invariant $\sigma=1$. \end{proposition}
\begin{proof} Let $Y$ be the unique Enriques surface with finite automorphism group of type III (respectively of type IV) in the sense of \cite{MR4009175}. It follows from \cite[Lemma 11.1]{MR4009175} that $Y$ has a complex model $\CY$ of type III (respectively of type IV) in the sense of \cite{MR914299}. From \cite[Proposition 3.3.2]{MR914299} (respectively from \cite[Proposition 3.4.2]{MR914299}) it follows that the universal K3 cover $\CX$ of $\CY$ is the Kummer surface $\mathrm{Km}(\CE \times \CE)$, where $\CE$ is the complex elliptic curve of $j$-invariant $j=1728$. Thus, the universal K3 cover $X$ of $Y$ is the Kummer surface $\mathrm{Km}(E \times E)$ where $E$ is the elliptic curve of $j$-invariant $j=1728$ over $k$, which is a supersingular elliptic curve in characteristic $p=3$.   \end{proof}
As a corollary we can identify the two surfaces from Theorem \ref{p3s1}.
\begin{corollary}\label{types} The two Enriques quotients of the supersingular K3 surface of Artin invariant $\sigma= 1$ over an algebraically closed field of characteristic $3$ are the unique Enriques surfaces of type III and type IV following the classification in \cite{MR4009175}. \end{corollary}
\newpage
\printbibliography

\end{document}